\documentclass[11pt]{amsart}
\usepackage{amsmath, amstext, amsbsy, amssymb}

\hoffset \voffset \oddsidemargin=55pt \evensidemargin=55pt
\numberwithin{equation}{section}
\def\beq{\begin{eqnarray}}
\def\eeq{\end{eqnarray}}
\def\beqs{\begin{eqnarray*}}
\def\eeqs{\end{eqnarray*}}

\def\mz{{\mathbb Z}}

\def\ind{\hbox{\rm ind}}
\def\ord{\hbox{\rm ord}}
\newfont{\df}{eufm10}

\newtheorem{theo}{Theorem}[section]
\newtheorem{theorem}[theo]{Theorem}
\newtheorem{defi}[theo]{Definition}

\newtheorem{lemma}[theo]{Lemma}

\newtheorem{proposition}[theo]{Proposition}

\newcommand{\N}{\mathbb N}
\newcommand{\Z}{\mathbb Z}

\def\mod{{\hbox{\rm mod}}}

\title[Minimal zero-sum sequences of length five]{Minimal zero-sum sequence of length five over finite cyclic groups of prime power order}
\author[L.-M. Xia]{Li-meng   Xia$^{1}$}
\author[Y.-L. Li]{Yuanlin Li$^{*,2}$}
\author[J.-T. Peng]{Jiangtao Peng$^3$}
\thanks{
*Corresponding author: Yuanlin Li, Department of Mathematics, Brock
University, St. Catharines, Ontario Canada L2S 3A1, Fax:(905)
378-5713;\\ E-mail: yli@brocku.ca (Y. Li) \\ }
\date{}
\begin{document}
\maketitle
\centerline{\small $^1$Faculty of Science, Jiangsu University, Zhenjiang, 212013, Jiangsu Pro., P.R. China}
\centerline{\small $^2$Department of Mathematics, Brock University, St. Catharines, ON, Canada L2S 3A1}
\centerline{\small $^3$College of Science, Civil Aviation University of China, Tianjin, 300300, P.R. China}

\def\abstractname{ABSTRACT}
\begin{abstract}
Let $G$ be a finite cyclic group. Every sequence $S$ of length $l$ over $G$ can be written in the form
$S=(x_1g)\cdot\ldots\cdot(x_lg)$ where $g\in G$ and $x_1, \ldots, x_l\in[1, \ord(g)]$, and the index $\ind(S)$
 of $S$ is defined to be the minimum of $(x_1+\cdots+x_l)/\ord(g)$ over all possible $g\in G$
 such that $\langle g \rangle =G$. Recently the second and the third authors determined the index of any minimal zero-sum sequence $S$ of length 5 over a cyclic group of a prime order where $S=g^2(x_2g)(x_3g)(x_4g)$. In this paper, we determine the index of any minimal zero-sum sequence $S$ of length 5 over a cyclic group of a prime power order.   It is shown that if $G=\langle g\rangle$ is a cyclic group of  prime power order  $n=p^\mu$  with $p \geq 7$ and $\mu\geq 2$,  and $S=(x_1g)(x_2g)(x_2g)(x_3g)(x_4g)$ with $x_1=x_2$ is a minimal zero-sum sequence with $\gcd(n,x_1,x_2,x_3,x_4,x_5)=1$, then  $\ind(S)=2$ if and only if  $S=(mg)(mg)(m\frac{n-1}{2}g)(m\frac{n+3}{2}g)(m(n-3)g)$ where $m$ is a positive integer such that $\gcd(m,n)=1$.
\end{abstract}

\vskip 3mm


\setcounter{section}{0}

\section{Introduction}

Throughout the paper $G$ is assumed to be a finite cyclic
group of order $n$ written additively.  Denote by  $\mathcal{F}(G)$, the free abelian monoid with basis $G$ and
elements of $\mathcal{F}(G)$ are called \emph{sequences} over $G$. A sequence of length $ l$ of
not necessarily distinct elements from $G$ can be written in the
form $S=(x_1g)\cdot \, \ldots \,\cdot (x_lg)$ for some $g\in G$.  $S$  is called a \emph{zero-sum
sequence} if the sum of $S$ is zero (i.e. $\sum_{ i=1}^l x_ig=0$). If $S$ is a zero-sum sequence, but no proper
nontrivial subsequence of $S$ has sum zero, then  $S$ is called a \emph{minimal zero-sum sequence}.  The index of a sequence $S$ over $G$ is defined as follows.

\begin{definition}
 For a sequence over $G$
 $$S=(x_1g)\cdot\ldots\cdot(x_lg), \,\,\, \mbox{where} \,\, 1\le x_1, \ldots, x_l \le \ord(g),$$
 the index of $S$ is defined by
 $\ind(S)=\min\{\| S \|_g \,|\,g\in G \, \mbox{with}\,\, G=\langle g \rangle\}$ where  $$\|S\|_g=\frac{x_1+\cdots+x_l}{\ord(g)}.$$
\end{definition}
\noindent Clearly, $S$ has sum zero if and only if $\ind(S)$ is an integer. There are also slightly different definitions of the index in the literature, but they are all equivalent (see Lemma 5.1.2 in \cite{Ge:09a}).

 The index of a sequence is a crucial invariant
in the investigation of (minimal) zero-sum sequences (resp. of
zero-sum free sequences) over cyclic groups. It was first addressed
by Kleitman-Lemke (in the conjecture \cite[page 344]{KL:89}),
used as a key tool by Geroldinger (\cite[page 736]{Ge:87}), and
then investigated by Gao \cite{Gao:00} in a systematical way. Since then it has received a great deal of attention (see for example \cite{CFS:99, CS:05, GG:09, GLPPW:11, Ge:09a, GeH:06, LP:13, LPYZ:10, PL:13,
P:04, SC:07, SX:13, SXL:14, X:13, XS:13, XY:10, Y:07, YZ:11}).

 A main focus of the investigation of index is to determine minimal zero-sum sequences of index 1. If $S$ is a minimal zero-sum sequence of length $|S|$ such that $|S|\leq 3$ or $|S| \geq \lfloor\frac{n}{2}\rfloor+2$, then
ind(S) = 1 (see \cite{CFS:99,SC:07,Y:07}). In contrast to that, it was shown that for each $l \mbox{ with }  5 \leq l \le \lfloor\frac{n}{2}\rfloor + 1$, there is a minimal zero-sum sequence $S$ of length $|S| = l$ with $\ind(S)\geq 2$ (\cite{SC:07, Y:07})  and that the
same is true for $l = 4$ and $\gcd(n, 6) \ne 1$ (\cite{P:04}). In  recent papers \cite{LP:13, LPYZ:10, XS:13},  the  authors proved that $\ind(S)=1$ if $|S|=4$ and $\gcd(n, 6) = 1$ when $n$ is a prime power or a product of two prime powers. When $n$ is a product of at least 3 prime powers, some partial results were  also obtained in \cite{SX:13, SXL:14, X:13}. However,  the general case is still open.


It was mentioned in \cite{PL:13}, in order to further investigate the index of a general minimal zero-sum sequence of length 4, it is helpful to determine the index of certain minimal zero-sum sequences of length 5.  It is routine to check that  if $S$ is a minimal zero-sum sequence over $G$ of length 5, then $1\le \ind(S) \le 2$.   Let $\mathsf h(S)$ be the maximal repetition  of an element in $S$. In \cite{PL:13}, the index of any minimal zero-sum sequence $S$ of length 5  over acyclic group of a prime order with $\mathsf h(S) \geq 2$ was completely determined.  In this paper, we continue the investigation on the index of minimal zero-sum sequences of length 5 over a cyclic group of prime power order.  When $ G$ is a cyclic group of  prime power order  $n=p^\mu$  with $p \geq 7$ and $\mu\geq 2$, and $S$ is a minimal zero-sum sequence over $G$  with $\mathsf h(S) \geq 2$,  we were able to determine completely the index of $S$.
Our main result is  as follows.

\begin{theorem}\label{maintheorem}
Let $G$ be a cyclic group of prime power order $n=p^\mu$ with $ p\geq 7$ and $\mu\geq2$, and  $S=(x_1g)\cdot(x_2g)\cdot(x_3g)\cdot(x_4g)\cdot(x_5g)$ be a minimal zero-sum sequence over $G$ with $\gcd(n,x_1,x_2,x_3,x_4,x_5)=1$ and $x_1=x_2$, then $\ind(S)=2$ if and only if  $S=(mg)(mg)(m\frac{n-1}{2}g)(m\frac{n+3}{2}g)(m(n-3)g)$ where $m$ is a positive integer such that $\gcd(m,n)=1$.
\end{theorem}

The paper is organized as follows. In the next section, we provide some preliminary results. In section 3, we state three main propositions and use them together with some preliminary results to give a proof for our main result. The proofs of the main propositions are given in the last section.

\section{Preliminaries}

We first prove some preliminary results which will be needed in the sequel. Let $G=\langle g\rangle$ be a cyclic group of order $n$. Suppose that  $S=(x_1g)\cdot \, \ldots \,\cdot (x_lg)$. Let  $ \|S\|_g' =\ord(g) \|S\|_g=\sum_{ i=1}^l x_i \in N_0$ and denote by $|x|_n$ the least positive residue of $x$ modulo $n$, where $n \in \N$ and $x\in \Z$.  Let $mS$ denote the sequence $(mx_1g)\cdot \, \ldots \,\cdot (mx_lg)$. Since $|g|=n$, we have  $ mS=(|mx_1|_ng)\cdot \, \ldots \,\cdot (|mx_l|_ng).$  We note that if $\gcd(n, m) = 1$, then the multiplication by $m$ is a group automorphism of $G$ and hence $\ind(S)=\ind(mS)$. Two sequences $S$ and $S'$ are called equivalent, denoted by $S\sim S'$, if $S=mS'$ for some $m$ with $\gcd(n, m) = 1$. Clearly, if $S\sim S'$, then $\ind(S)=\ind(S')$. For all real numbers $a < b$, define $[a,b]=\{k\in\mz | a\leq k\leq b\}$.

From now on we always assume that  $G=\langle g\rangle$ is a cyclic group of order $n=p^\mu $ with $ p\geq 7$ and $\mu\geq2$, and $S=(x_1g)\cdot \, \ldots \,\cdot (x_lg)$ is a minimal zero-sum sequence over $G$. We remark that if  $\gcd(n,x_1,x_2,x_3,x_4,x_5)=p^w >1$, let $h=p^wg$ and $x'_i= x_i/p^w$ for $1\leq i\leq 5$. Then a minimal zero-sum sequence  $S$ over $G$ can be rewritten as follows:
$$S=(x'_1h)\cdot \, \ldots \,\cdot (x'_lh),$$
which is a minimal zero-sum sequence over $H=\langle h\rangle$. Note that $H$ is a cyclic group of prime power order with $|H| | |G|$ and $\gcd(n,x'_1,x'_2,x'_3,x'_4,x'_5)=1$. In what follows we may always assume that $\gcd(n,x_1,x_2,x_3,x_4,x_5)=1$. Let $U(n) $ denote the unit group of $n=p^\mu$, i.e. $U(n)=\{ m| 1\leq m \leq n-1, \gcd(m, n)=1\}$. It is well known that $1+t\alpha\in U(n)$ for any $1<\alpha=p^\lambda<n$.

\begin{lemma}\label{L2.1} 
Let $n=p^{\mu}$ with $\mu\geq 2$ and  $\alpha=\frac{n}{p}$. If $v$ is an integer such that $1\leq v\leq n-1$ and $\gcd(v, p)=1$, then there exists $y=1+t\alpha\in U(n)$ such that $|vy|_n<\frac{n}{p}$.
\end{lemma}

\begin{proof}
Note that  $y=1+t\alpha\in U(n)$ for any $0\leq t\leq p-1$. There exists a $ y_1=1+t\alpha$ such that $vy_1 <n$, but $vy > n$ where $y=1+(t+1)\alpha$. Thus $|vy|_n< \alpha=\frac{n}{p}$.
\end{proof}

The following proposition is a generalization of \cite[Proposition 2.2]{PL:13}.

\begin{proposition}\label{P2.2}
Let   $G=\langle g\rangle$ be a cyclic group of  order $n=p^\mu$ with $p\ge 5$. If $S=g^2\cdot (\frac{p^{\mu}-1}{2}g)\cdot (\frac{p^{\mu}+3}{2}g)\cdot ((p^{\mu}-3)g)\in \mathcal{F}(G)$, then $\ind(S)=2$.
\end{proposition}
\begin{proof}
This lemma was proved in \cite{PL:13} for $n=p$ and the same proof works for the prime power case.
\end{proof}


\begin{proposition}\label{P2.3}
Let   $G$ be a cyclic group of  prime power order $n=p^{\mu}$ with $p\geq 5$ and $\mu\geq2$, and  $S\in \mathcal F(G)$ be a minimal zero-sum sequence of length $5$. If $\mathsf h(S) \ge 3$, then $\ind(S)=1$.
\end{proposition}

\begin{proof}
If $\mathsf h(S)\ge 4$, then we may assume that $x_1=x_2=x_3=x_4=\gcd(n,x_1)$. So we have that  $\|S\|_g'=\sum_{i=1}^5x_i=4x_1+x_5 < 2n$. Therefore $\ind(S)=1$.

If $\mathsf h(S)=3$, then we may assume that $x_1=x_2=x_3=\gcd(n,x_1)$. If $x_1=1$, then $x_4, x_5\leq n-4$ and $\|S\|_g'=\sum_{i=1}^5x_i\leq 3+2(n-4)<2n$, so we infer that  $\|S\|_g'=\sum_{i=1}^5x_i=n$ and $\ind(S)=1$.

If $x_1=p^s>1$, then $\gcd(n,x_4)=1$. Let $\alpha=p^{\mu-1}$. Then by Lemma~\ref{L2.1}  there exists $y=1+t\alpha\in U(n)$ such that $|yx_4|_n<\alpha$. Since $x_1\leq \alpha$ and $y$ fixes $x_1$ (i.e. $|yx_1|_n=|x_1|_n$), we have $\|yS\|_g'=|yx_1|_n+\cdots+|yx_5|_n < 4\alpha+x_5<2n$, so $\ind(S)=1$ and we are done.
\end{proof}

By using a computer search, we were able to verify the following proposition.

\begin{proposition}\label{small n}
Let   $G=\langle g\rangle$ be a cyclic group of  prime power order $n=p^{\mu}$ with $p\geq 5$ and $\mu\geq2$, and  $S=g^2\cdot(x_3g)\cdot(x_4g)\cdot(x_5g)$ be a minimal zero-sum sequence of length $5$. If $n < 289$, then $\ind(S)=2 $ if and only if $S=g^2\cdot (\frac{n-1}{2}g)\cdot (\frac{n+3}{2}g)\cdot ((n-3)g)$.
\end{proposition}

\section{Main Propositions}

In this section we present three main propositions and then use them together with some preliminary results to provide a proof for our main theorem. In terms of Proposition~\ref{P2.3},  from now on we may always assume that  $\mathsf h(S)=2$ and  $x_1=x_2=\gcd(n,x_1)$. The following proposition takes care of the case when $x_1>1$.

\begin{proposition}\label{P3.1}
Let   $G=\langle g\rangle$ be a cyclic group of  prime power order $n=p^{\mu}$ with $p\geq 7$ and $\mu\geq2$, and  $S=(x_1g)\cdot(x_2g)\cdot(x_3g)\cdot(x_4g)\cdot(x_5g)$ be a minimal zero-sum sequence of length $5$. If $x_1(=x_2=\gcd(n,x_1)) >1 $, then $\ind(S)=1 $.
\end{proposition}

Next we assume that $x_1=x_2=\gcd(n,x_1)=1$ and   $S=g^2\cdot(x_3g)\cdot(x_4g)\cdot(x_5g)$ such that $2+x_3+x_4+x_5=2n$. The following two propositions are crucial to prove the main theorem.

\begin{proposition}\label{P3.2}
Let $G$ be a cyclic group of order $n$ with $n=p^\mu\geq 289, p\geq 7$ and $\mu\geq2$. Let $S=g^2\cdot(cg)\cdot((n-b)g)\cdot((n-a)g)$ be a minimal zero-sum sequence with $4\leq a\leq b<c<\frac{n}{2}$ and $c=b+a-2$. Then $\ind(S)=1$.
\end{proposition}

\begin{proposition}\label{P3.3}
Let $G$ be a cyclic group of order $n$ with $n=p^\mu\geq 289, p\geq 7$ and $\mu\geq2$. Let $S=g^2\cdot(cg)\cdot((n-b)g)\cdot((n-a)g)$ be a minimal zero-sum sequence with $3=a\leq b<b+1=c<\frac{n}{2}$. Then $\ind(S)=2$ if and only if $c=\frac{n-1}{2}$.
\end{proposition}

We are now ready to present a proof for our main result.

\noindent{\bf Proof of Theorem~\ref{maintheorem}.}

In terms of Proposition~\ref{P2.3}, we may assume that  $\mathsf h(S)=2$ and  $x_1=x_2=\gcd(n,x_1)$. If $x_1>1$, then the result follows from Proposition~\ref{P3.1}. Next we assume
$x_1=x_2=\gcd(n,x_1)=1$ and  $S=g^2\cdot(x_3g)\cdot(x_4g)\cdot(x_5g)$. As discussed in \cite{PL:13}, we may assume that $S=g^2\cdot(cg)\cdot((n-b)g)\cdot((n-a)g)$ with $3 \leq a\leq b< c<\frac{n}{2}$ and $c=a+b-2$, for otherwise, it is easy to show that $\ind(S)=1$. If $n< 289$, the result follows from Proposition~\ref{small n}. So we may assume that $n\geq 289$. Now the result follows immediately from Propositions~\ref{P3.2} and \ref{P3.3}, and we are done.

\section{Proofs For Main Propositions}

We now give the proofs for the three crucial propositions stated in previous section. The proofs will be presented in the following three subsections.

\subsection{Proof of Proposition~\ref{P3.1}}.

\bigskip
Let $x_1=p^s>1$ and $\alpha'=\frac{n}{p^s}=\frac{n}{x_1}$. We first assume that $\gcd(n,x_i)>1$ for some $i=3,4,5$, say $\gcd(n,x_3)>1$. We divide the proof into 3 cases.

{\bf Case 1}. If $x_1>\gcd(n,x_3)>1$, let $x'_3=\frac{x_3}{\gcd(n,x_3)}$, $n'=\frac{n}{\gcd(n,x_3)}$ and $\alpha'=\frac{n'}{p}$. Then by Lemma~\ref{L2.1}, we may find a $y=1+t\alpha'\in U(n')$ such that $|yx'_3|_{n'}<\alpha'$. Thus $|yx_3|_n=\gcd(n,x_3)|yx'_3|_{n'}< \gcd(n,x_3)\alpha'=\frac{n}{p}$. Since $x_1>\gcd(n,x_3)$, $y$ fixes $x_1=x_2$, so by multiplying $S$ with $y$, we may assume that $x_1, x_2, x_3\leq \frac{n}{p}$. Since  $\gcd(n,x_1,x_2,x_3,x_4,x_5)=1$, we have $\gcd(n,x_4)=1$. Thus by Lemma~\ref{L2.1}, there exists $y'=1+t\frac{n}{p}\in U(n)$ such that $|y'x_4|_n<\frac{n}{p}$. Since $y'$ fixes $x_1, x_2$ and $x_3$,  again we may assume that  $x_i\leq \frac{n}{p}$ for $i=1,2,3,4$. Thus $ \|S\|_g'=\sum_{i=1}^5x_i\leq \frac{4n}{p}+x_5<2n$, so $\ind(S)=1$ and we are done.

 {\bf Case 2}. If $x_1<\gcd(n,x_3)$, then there exists $u\in U(n)$ such that $|ux_3|_n=\gcd(n,x_3)$. So we may assume that $x_3=\gcd(n,x_3)$. Note that $x_1$ may not be $\gcd(n,x_1)$ any longer. However, as in Case 1 we may find $y=1+t\alpha'$, where $\alpha'=\frac{n}{\gcd(n,x_3)}$,  such that $|yx_1|_n< \alpha'\leq \frac{n}{p}$. Since $y$ fixes $x_3$, we may assume that $x_1, x_2, x_3\leq \frac{n}{p}$.
Since $\gcd(n,x_4)=1$,  there exists $y'=1+t\frac{n}{p}\in U(n)$ such that $|y'x_4|_n<\frac{n}{p}$. Since $y'$ fixes $x_1, x_2$ and $x_3$, so we may assume that  $x_i\leq \frac{n}{p}$ for $i=1,2,3,4$. Thus $\sum_{i=1}^5x_i\leq \frac{4n}{p}+x_5<2n$, so $\ind(S)=1$.

{\bf Case 3}. If $x_1=\gcd(n,x_3)$, let $x_3=n-kx_1=n-w$. Since $S$ is a minimal zero-sum sequence, we have $k\geq3$. (Otherwise, $x_1+x_3$ or $x_1+x_2+x_3$ has zero-sum.) If $w<\frac{n}{2}$, then $2w<n$ and $|2x_3|_n=n-2w$ and $|2x_1|_n=2x_1$. We may replace $S$ by the following equivalent sequence:
\beqs S'&=&(2x_1g)\cdot(2x_2g)\cdot((n-w')g)\cdot(|2x_4|_ng)\cdot(|2x_5|_ng).\eeqs
By repeating this process, we may assume that $w>\frac{n}{2}$. As before, since $\gcd(n,x_4)=1$, we may assume that $x_4<\frac{n}{p}$. Since $k\geq 3, p>6$ and $kp-2p-2k=(k-2)(p-2)-4>0$, $\frac{k-2}{k}w-x_4>\frac{k-2}{k}\frac{n}{2}-\frac{n}{p}=\frac{n}{2kp}(kp-2p-2k)>0$, so we have
\beqs  \|S\|_g'=\sum_{i=1}^5x_i&=&n-(k-2)x_1+x_4+x_5\\
&=&n+x_5-(\frac{k-2}{k}w-x_4)<2n.\eeqs
Therefore,  $\ind(S)=1$ and we are done.

 Next we assume that $\gcd(n,x_i)=1$ for all $i=3,4,5$. Let $x_1=p^s>1$ and $\alpha'=\frac{n}{p^s}$. As before, we may use a unit $1+t\alpha'(0\leq t\leq x_1-1)$ to move $x_3$ so that $x_3<\alpha'$. Thus we may assume that $x_3<\alpha'$. Let $\beta=\frac{n}{p}\geq\max(x_1,\alpha')$.  Then $y_1=1+t_1x_3^{-1}\beta\in U(n)$ and $y_2=1+t_2x_3^{-1}\beta\in U(n)$ for $t_1 \not=t_2\in [0,p-1]$ are distinct units. Note that every such $y$ fixes $x_1$. 
 Note also that $|(1+tx_3^{-1}\beta)x_i|_n= |x'_i+ (n-t'\beta)|_n \leq x'_i +(n-t'\beta) $ where $ 1\leq t'\leq p$ and $x'_i= x_i \,\mod\,\beta $. If $|(1+t_1x_3^{-1}\beta)x_i|_n \not=|(1+t_2x_3^{-1}\beta)x_i|_n $, then $|x'_i+ (n-t_1'\beta)|_n \not= |x'_i+ (n-t_2'\beta)|_n $, implying $t_1'\not =t'_2$. Thus $\sum_{t=0}^6 |(1+tx_3^{-1}\beta)x_i|_n =\sum_{t=0}^6|x'_i+ (n-t'\beta)|_n \leq 7x'_i +\sum_{t'=1}^7(n-t'\beta)= 7x'_i +7n -28\beta $.  We now compute the following sum
\beqs \sum_{i=1}^5\sum_{t=0}^6|(1+tx_3^{-1}\beta)x_i|_n&=& 7x_1+7x_2+\sum_{t=0}^6(x_3+t\beta)
+\sum_{i=4,5}\sum_{t=0}^6|(1+tx_3^{-1}\beta)x_i|_n\\
&\leq&14x_1+7x_3+21\beta+7x_4'+7x_5'+2\sum_{t=0}^6(n-(1+t)\beta)\\
&=&7(2x_1+x_3+x_4'+x_5')+14n-35\beta.\eeqs
Since $2x_1+x_3+x_4'+x_5'<5\beta$, we have
\beqs \sum_{i=1}^5\sum_{t=0}^6|(1+tx_3^{-1}\beta)x_i|_n<14n+35\beta-35\beta=14n.\eeqs
Then there must exist $t\in[0,6]$ such that
\beqs \sum_{i=1}^5|(1+tx_3^{-1}\beta)x_i|_n<2n.\eeqs
So \beqs \sum_{i=1}^5|(1+tx_3^{-1}\beta)x_i|_n=n,\eeqs
 and thus $\ind(S)=1$. This completes the proof.

\subsection{Proof of Proposition~\ref{P3.2}}.

\bigskip
In this subsection we will prove Proposition~\ref{P3.2} through the following lemmas.   Let $s=\lfloor\frac{b}{a}\rfloor$. Recall that from now on we always assume that $n\geq 289$.

\begin{lemma}\label{L4.1}
Let $S=g^2\cdot(cg)\cdot((n-b)g)\cdot((n-a)g)$ be the same sequence described in Proposition~\ref{P3.2}. If there exist integers $M$ and $k$ such that $\gcd(n,M)=1$, $\frac{kn}{c}<M<\frac{kn}{b}$ and $Ma<n$, then $\ind(S)=1.$
\end{lemma}

\begin{proof}
Note that
\beqs  \|MS\|_g'&=&M+M+|Mc|_n+|M(n-b)|_n+|M(n-a)|_n\\
&\leq &M+M+Mc-kn+(kn-Mb)+(n-Ma)\\
&=&n+M(1+1+c-b-a)=n.\eeqs
We conclude that $\ind(S)=1$.
\end{proof}

\begin{lemma}\label{L4.2}
If $S=g^2\cdot(cg)\cdot((n-b)g)\cdot((n-a)g)$
is a minimal zero-sum sequence such that  $3\leq a\leq b<c<\frac{n}{2}$, then $\ind(S)=1$ provided that one of the following two conditions holds:\\
(1) $\frac{s(a-2)n}{bc}> 2.$\\
(2) $\frac{(s-1)(a-2)n}{bc}> 1$ and $\frac{n}{c}< p-2.$
\end{lemma}

\begin{proof}
(1). If  $\frac{s(a-2)n}{bc}> 2$, then there exist two integers $M_1, M_2 \in [\frac{sn}{c}, \frac{sn}{b}]$, so one of them, say $M$, is co-prime to $n$. Since $Ma<n$, by Lemma~\ref{L4.1}
 we have  $\ind(S)=1$.

(2). If $\frac{(s-1)(a-2)n}{bc}> 1$, then we can find two integers $M_1$ and $M_2$ such that $M_1 \in [\frac{(s-1)n}{c}, \frac{(s-1)n}{b}]$ and $M_2 \in [\frac{sn}{c}, \frac{sn}{b}]$. We may assume that $\frac{(s-1)n}{b}-M_1 \leq 1 $ and $M_2- \frac{sn}{c} \leq 1$; for otherwise, by the proof in (1) we  infer $\ind(S)=1$. Note that

$$M_2-M_1 =(M_2- \frac{sn}{c})+(\frac{sn}{c}-\frac{(s-1)n}{b}) +(\frac{(s-1)n}{b}-M_1 ) \leq 2+\frac{n}{c}< p.$$

Thus one of $M_1$ and $M_2$ is  co-prime to $n$, so  by Lemma~\ref{L4.1}
 we have  $\ind(S)=1$.
\end{proof}

\begin{lemma}\label{L4.3}
If the sequence $$S=g^2\cdot(cg)\cdot((n-b)g)\cdot((n-a)g),$$
is a minimal zero-sum sequence such that  $6\leq a\leq b<c<\frac{n}{2}$ and $s\geq 4$, then $\ind(S)=1$.
\end{lemma}
\begin{proof}
We divide the proof into the following 3 cases.

{\bf Case 1.} $\frac{n}{c}\geq 4$. Since $\frac{s(a-2)n}{bc}> \frac{s(a-2)n}{(s+1)ac}\geq (\frac{4}{5})(\frac{4}{6})4 >2$, by Lemma~\ref{L4.2} we infer $\ind(S)=1$.

{\bf Case 2.} $3 \leq \frac{n}{c}< 4$. Then $\frac{(s-1)(a-2)n}{bc}> \frac{(s-1)(a-2)n}{(s+1)ac}\geq (\frac{3}{5})(\frac{4}{6})3 >1 $. Since $\frac{n}{c}<4< p-2$,  it follows from Lemma~\ref{L4.2} that $\ind(S)=1$.

{\bf Case 3.} $2 < \frac{n}{c}< 3$. If $\frac{n}{c}< 3 <\frac{n}{b}$, then $3a <n$ and thus by Lemma~\ref{L4.1} $\ind(S)=1$. So we may assume that $2 < \frac{n}{c}<\frac{n}{b} < 3$.
If $s\geq 7$, then  $\frac{(s-1)(a-2)n}{bc}> \frac{(s-1)(a-2)n}{(s+1)ac} \geq (\frac{6}{8})(\frac{4}{6})2 = 1$. By Lemma~\ref{L4.2}  $\ind(S)=1$. If $4\leq s \leq 6$, $n <3b< 3(s+1)a\leq 21 a$, so $a\geq 14$. Again, $\frac{(s-1)(a-2)n}{bc}> \frac{2(s-1)(a-2)}{(s+1)a} \geq (\frac{3}{5})(\frac{12}{14})2 > 1$. By Lemma~\ref{L4.2} we infer $\ind(S)=1$.

\end{proof}

\begin{lemma}\label{L4.4}
If $S=(g)\cdot(g)\cdot(cg)\cdot((n-b)g)\cdot((n-a)g)$
is a minimal zero-sum sequence such that  $6\leq a\leq b<c<\frac{n}{2}$ and $s\leq 3$, then $\ind(S)=1$.
\end{lemma}

\begin{proof}
We  divide the proof into two cases.

{\bf Case 1.} $\frac{n}{c}\geq 5$. If $s\geq 2$, then $\frac{s(a-2)n}{bc}> \frac{s(a-2)n}{(s+1)ac}> (\frac{2}{3})(\frac{4}{6})5 >2$, so  by Lemma~\ref{L4.2},  $\ind(S)=1$.

Next assume that $s=1$.  If $\frac{n}{c}\geq 6$, then $\frac{(a-2)n}{bc}> \frac{(a-2)n}{2ac}\geq 2$, so again we infer $\ind(S)=1$. If  $5\leq \frac{n}{c}< 6$, then $3a-3\geq c > \frac{n}{6}$, so $a \geq 18$. Thus $\frac{(a-2)n}{bc}> \frac{(a-2)n}{2ac}\geq (\frac{16}{36})5 >  2$, so $\ind(S)=1$.

{\bf Case 2.} $\frac{n}{c}< 5$. We may assume that $m<\frac{n}{c}<\frac{n}{b}<m+1\, (*)$ with $m\in[2,4]$ (for otherwise, it is easy to show that $\ind(S)=1$).

{\bf Subcase 2.1.} $s=1$. 
If $m\geq 3$, we have
\beqs \frac{n}{b}-\frac{n}{c}=\frac{(a-2)n}{bc}> \frac{3(a-2)}{2a}\geq 1,\eeqs
giving a contradiction to $(*)$.

Next assume  $m=2$. If $b\leq 2a-4$, then $\frac{n}{b}-\frac{n}{c}=\frac{(a-2)n}{bc}\geq\frac{n}{2c}>1$, giving a contradiction. If $b\geq 2a-3$, then $n\geq 2c+1=6a-9$. Since
 $\frac{2n}{b}-\frac{2n}{c}= \frac{2(a-2)n}{bc}\geq \frac{4(a-2)}{2a-1} > 1$, we have $\frac{2n}{c}<5<\frac{2n}{b}$. Note that $3a-3\geq c > \frac{n}{3}$, so $a \geq 34$. Thus $5a<6a-9<n$, so by Lemma~\ref{L4.1} $\ind(S)=1.$




{\bf Subcase 2.2.} $s=2$.
If $m=2$, then $2<\frac{n}{c}<\frac{n}{b}<3$ and thus $a>\frac{n}{9}>32$. Since $\frac{2n}{b}-\frac{2n}{c}= \frac{2(a-2)n}{bc}\geq \frac{4(a-2)}{3a-1} > 1$, we have $\frac{2n}{c}<5<\frac{2n}{b}$, implying $\ind(S)=1$.


Next assume $m = 3$. Thus  $a> \frac{b}{3}>\frac{n}{12}>24$. So $\frac{2n}{b}-\frac{2n}{c}>\frac{2(a-2)n}{3ac}> \frac{(2 \times 22)n}{(3\times 24)c}=\frac{11n}{18c}$. If $\frac{n}{c}>\frac{36}{11}$, then  $\frac{2(a-2)n}{bc}> 2$, implying $\ind(S)=1$. Next assume that  $\frac{n}{c}\leq\frac{36}{11}$. Then   $\frac{3n}{c}\leq\frac{108}{11}< 10$. Since $\frac{3n}{b}-\frac{3n}{c}\geq \frac{3(a-2)n}{3ac} > 1$. We have $\frac{3n}{c}< 10 < \frac{3n}{b}$. We next show that $10a < n$ and thus $\ind(S) =1$ (by Lemma~\ref{L4.1}). Note that $3< \frac{n}{c} <\frac{n}{b} < 4$ and $a>24$. Then $a \leq \frac{24}{22}(a-2)\leq \frac{24}{22}(c-b)< \frac{24}{22}(\frac{n}{3}-\frac{n}{4})=\frac{n}{11}$, so $10a < n$ as required.

If $m= 4$, then $a\geq 20$ and thus $\frac{2(a-2)n}{bc} > \frac{8(a-2)}{3a}> 2$. By Lemma~\ref{L4.2} $\ind(S)=1$.

{\bf Subcase 2.3.} $s=3$.  If $m\geq 4$, then $\frac{3n}{b}-\frac{3n}{c}>\frac{3(a-2)n}{4ac}>2$. It follows from Lemma~\ref{L4.2} that $\ind(S)=1$. If $m= 3$, then $\frac{2n}{b}-\frac{2n}{c}>\frac{2(a-2)n}{4ac}>1$, so by Lemma~\ref{L4.2} $\ind(S)=1$.

Next we assume that $m=2$ and $ 2< \frac{n}{c}<\frac{n}{b}<3$. If $b\leq 4a-8$, then $\frac{2n}{b}-\frac{2n}{c}\geq \frac{2(a-2)n}{4(a-2)c}>1$, so $\ind(S)=1$. Now assume that $ 4a-7 \leq b\leq 4a-1$. Thus $c\geq 5a-9$, so $n\geq 2c+1\geq 10a-17 > 9a$
(as $a>\frac{n}{12}\geq 24$). Since $\frac{3n}{b}-\frac{3n}{c}>\frac{3(a-2)}{2a}>1$, we have $6<\frac{3n}{c}< M<\frac{3n}{b}<9$.  If $M=8$, since $8a<n$, we have $\ind(S)=1$.
So we may assume $6<\frac{3n}{c}< 7<\frac{3n}{b}<8$. Thus $ \frac{2n}{c}< \frac{14}{3}< 5$. If $ \frac{2n}{c}<  5< \frac{2n}{b} $, since $5a<n$, we have $\ind(S)=1$. Thus $ 4<\frac{2n}{c}< \frac{2n}{b} < 5 .$  Since  $\frac{4n}{b}-\frac{4n}{c}>1$, we have $ 8<\frac{4n}{c}< 9<\frac{4n}{b} < 10$. Recall that $9a<n$, so by Lemma~\ref{L4.1}  $\ind(S)=1$ and we are done.

\end{proof}

\begin{lemma}\label{L4.5}
If  $$S=g^2\cdot(cg)\cdot((n-b)g)\cdot((n-a)g),$$
is a minimal zero-sum sequence such that $a\in[4,5]$, $ a\leq b<c<\frac{n}{2}$ and $s\leq 10$, then $\ind(S)=1$.
\end{lemma}

\begin{proof}
Note that $b \leq sa+4$ and $c\leq (s+1)a+2$. If $s\leq 3$, then $\frac{n}{c} \geq \frac{289}{(s+1)a+2} \geq \frac{289}{22} > 13$. Thus  $\frac{s(a-2)n}{bc}>\frac{13s(a-2)}{(s+1)a}\geq   \frac{13\times 2}{8} > 2.$  If $s\geq 4$, then $\frac{n}{c} \geq \frac{289}{(s+1)a+2} \geq \frac{289}{57} > 5$. Thus  $\frac{s(a-2)n}{bc}>\frac{5s(a-2)}{(s+1)a}\geq \frac{5\times 4\times 2}{5\times 4} = 2.$ It follows from Lemma~\ref{L4.2} that $\ind(S)=1$.
\end{proof}

\begin{lemma}\label{L4.6}
If  $S=g^2\cdot(cg)\cdot((n-b)g)\cdot((n-a)g)$
is a minimal zero-sum sequence such that $a\in[4,5]$, $ a\leq b<c<\frac{n}{2}$ and $s\geq 11$, then $\ind(S)=1$.
\end{lemma}

\begin{proof}
We divide the proof into two cases according to $a=5$ and $a=4$.

{\bf Case 1.}  $a=5$. In terms of Lemma~\ref{L4.2}, we may assume that  $\frac{s(a-2)n}{(s+1)ac} \leq 2$. Thus $\frac{3sn}{5(s+1)c} \leq 2$, so $\frac{n}{c} \leq \frac{10(s+1)}{3s} \leq \frac{120}{33} < 4 <p-2$. Since  $\frac{(s-1)(a-2)n}{bc} >\frac{(s-1)(a-2)n}{(s+1)ac} > \frac{10\times 3\times 2}{12\times 5} =1$, by  Lemma~\ref{L4.2} we have $\ind(S)=1.$


{\bf Case 2.}  $a=4$.  In terms of Lemma~\ref{L4.2}, we may assume that $\frac{s(a-2)n}{(s+1)ac} \leq 2$. Thus $\frac{n}{c} \leq \frac{8(s+1)}{2s} \leq \frac{96}{22} < 5 \leq p-2$.

If $\frac{n}{c}\geq 3$, then $\frac{(s-1)(a-2)n}{bc} >\frac{2(s-1)n}{(s+1)4c} > \frac{2\times 10\times 3}{12\times 4} >1$, by  Lemma~\ref{L4.2} we have $\ind(S)=1.$ Next as before we may assume $2<\frac{n}{c} <\frac{n}{b} <3$.




Since $2<\frac{n}{b}<3$, we have $b>\frac{n}{3}>96$ and $b\geq 97$. Assume that $n=2c+j=2b+j+4$. If $j\geq 17$, we have $2b^2+(j-10)b-7(j+4)-(2b^2+4b)=(j-14)b-7(j+4)=j(b-7)-14(b+2)=(j-14)(b-7)-14\times 9\geq 3\times 90-14\times 9=144>0$ and
\beqs \frac{(s-1)n}{b}-\frac{(s-1)n}{c}=\frac{2(s-1)n}{bc}\geq\frac{(b-7)(2b+j+4)}{2b(b+2)}=\frac{2b^2+(j-10)b-7(j+4)}{2b^2+4b}>1.\eeqs
By Lemma~\ref{L4.2} we have $\ind(S)=1$.

Next assume that $j\leq 15$ ($j$ is odd as $n$ is odd). Let $n=(j+3)t+r, r\in[0,j+2]$.

First we claim that $r< t$. Since $t=\lfloor\frac{n}{j+3}\rfloor\geq\lfloor\frac{n}{18}\rfloor\geq 16$ and $r\leq 17$. If  $r=17=t+1$, we have
$n=18\times 16+17=305$, which is not a prime power. If $r=17=t$, then $n=18\times 17+17=17\times 19$, which is  not a prime power. If $r=16=t$, then $n=18\times 16+16=16\times 19$ ,  which is  not a prime power. Hence our claim holds.

{\bf Subcase 2.1.} If $j=1$, let $n=5l+k$ with $k\in[1,4]$, so $l\geq 57$ and $\gcd(n,l)=1$. Then \beqs S&=&(g)\cdot(g)\cdot(\frac{n-j}{2}g)\cdot(\frac{n+j+4}{2}g)\cdot((n-4)g)\\
&\sim&(2g)\cdot(2g)\cdot((n-1)g)\cdot(5g)\cdot((n-8)g)\quad (\hbox{\rm multiplied by\;}2 )  \\
&\sim&(4lg)\cdot(4lg)\cdot((3l+k)g)\cdot((5l-k)g)\cdot((4l+4k)g)\quad (\hbox{\rm multiplied by\;} 2l)\\
&\sim&((l+k)g)\cdot((l+k)g)\cdot((l-k)g)\cdot((l+3k)g)\cdot((l-3k)g)\quad (\hbox{\rm multiplied by\;} n-1).
\eeqs
Since  $(l+k)+(l+k)+(l-k)+(l+3k)+(l-3k)=5l+k=n$, we have  $\ind(S)=1$.

{\bf Subcase 2.2.}  If $j=3$, let $n=7l+k$ with $k\leq 6$, then $l\geq 41> k$ and $\gcd(n,l)=1$. If $k>0$, then \beqs S&=&(g)\cdot(g)\cdot(\frac{n-j}{2}g)\cdot(\frac{n+j+4}{2}g)\cdot((n-4)g)\\
&\sim&(2g)\cdot(2g)\cdot((n-3)g)\cdot(7g)\cdot((n-8)g)\quad (\hbox{\rm multiplied by\;}2)   \\
&\sim&(6lg)\cdot(6lg)\cdot((5l+2k)g)\cdot((7l-2k)g)\cdot((4l+4k)g)\quad (\hbox{\rm multiplied by\;} 3l)\\
&\sim&((l+k)g)\cdot((l+k)g)\cdot((2l-k)g)\cdot(3kg)\cdot((3l-3k)g)\quad (\hbox{\rm multiplied by\;} n-1).
\eeqs
Since $(l+k)+(l+k)+(2l-k)+3k+(3l-3k)=7l+k=n$, we have $\ind(S)=1$.

If $k=0$, then $n$ is a power of $7$.
\beqs S&=&(g)\cdot(g)\cdot(\frac{n-j}{2}g)\cdot(\frac{n+j+4}{2}g)\cdot((n-4)g)\\
&\sim&(2g)\cdot(2g)\cdot((n-3)g)\cdot(7g)\cdot((n-8)g)\quad (\hbox{\rm multiplied by\;}2)   \\
&\sim&(6(l-1)g)\cdot(6(l-1)g)\cdot((5l+9)g)\cdot((7l-21)g)\cdot((4l+24)g)\\
&&\quad (\hbox{\rm multiplied by\;} 3(l-1))\\
&\sim&((l+6)g)\cdot((l+6)g)\cdot((2l-9)g)\cdot(21g)\cdot((3l-24)g)\\
&&\quad (\hbox{\rm multiplied by\;} n-1),
\eeqs
where $3l-24 =3(l-8)>0 $. Since $(l+6)+(l+6)+(2l-9)+21+(3l-24)=7l=n$, again we have $\ind(S)=1$.

{\bf Subcase 2.3.}
If $\gcd(n,t)=1$ and $j\geq 5$,   we have
\beqs S&=&(g)\cdot(g)\cdot(\frac{n-j}{2}g)\cdot(\frac{n+j+4}{2}g)\cdot((n-4)g)\\
&\sim&(2g)\cdot(2g)\cdot((n-j)g)\cdot((j+4)g)\cdot((n-8)g)\quad (\hbox{\rm multiplied by\;}2)   \\
&\sim&(2tg)\cdot(2tg)\cdot((3t+r)g)\cdot((t-r)g)\cdot(((j-5)t+r)g)\quad (\hbox{\rm multiplied by\;} t).
\eeqs
Since $2t+2t+(3t+r)+(t-r)+((j-5)t+r)=(j+3)t+r=n$, we have $\ind(S)=1$.

If $\gcd(n,t)>1$ and $j> 5$, we have
\beqs S&=&(g)\cdot(g)\cdot(\frac{n-j}{2}g)\cdot(\frac{n+j+4}{2}g)\cdot((n-4)g)\\
&\sim&(2g)\cdot(2g)\cdot((n-j)g)\cdot((j+4)g)\cdot((n-8)g)\quad (\hbox{\rm multiplied by\;}2)   \\
&\sim&(2(t+1)g)\cdot(2(t+1)g)\cdot((3t+r-j)g)\cdot((t-r+j+4)g)\cdot(((j-5)t+r-8)g)\\
&&\quad (\hbox{\rm multiplied by\;} (t+1)),
\eeqs
where $(j-5)t+r-9 \geq t+r-9>0$. Since $2(t+1)+2(t+1)+(3t+r-j)+(t-r+j+4)+((j-5)t+r-8= (j+3)t+r=n$, we infer $\ind(S)=1$.

If $\gcd(n,t)>1$ and $j=5$, we have
\beqs S&=&(g)\cdot(g)\cdot(\frac{n-j}{2}g)\cdot(\frac{n+j+4}{2}g)\cdot((n-4)g)\\
&\sim&(2g)\cdot(2g)\cdot((n-j)g)\cdot((j+4)g)\cdot((n-8)g)\quad (\hbox{\rm multiplied by\;}2)   \\
&\sim&(2(t-1)g)\cdot(2(t-1)g)\cdot((3t+r+j)g)\cdot((t-r-j-4)g)\cdot(((j-5)t+r+8)g)\\
&&\quad (\hbox{\rm multiplied by\;} (t-1)),
\eeqs
where $t-r-j-9 \geq t+r-9>0$. Since $2(t-1)+2(t-1)+(3t+r+j)+(t-r=j-4)+((j-5)t+r+8= (j+3)t+r=n$, again we infer $\ind(S)=1$.

\end{proof}

Now Proposition~\ref{P3.2} follows immediately from Lemmas~\ref{L4.3}, \ref{L4.4}, \ref{L4.5} and \ref{L4.6}.

\subsection{Proof of Proposition~\ref{P3.3}}.

\bigskip

In this subsection, we will provide a proof for Proposition~\ref{P3.3}. In terms of Proposition~\ref{P2.2}, from now on we may always assume that $S =g^2\cdot(cg)\cdot((n-b)g)((n-3)g)$ where $a=3\leq b< c \leq \frac{n-3}{2}$ and $n=p^{\mu}\geq 289 $ with $p\geq 7$ and $\mu \geq 2$ and show that $\ind(S)=1$.

\begin{lemma}\label{L4.7}
If $S$ is a  minimal zero-sum sequence such that  $\frac{n}{c}>4$,
then  $\ind(S)=1$.
\end{lemma}

\begin{proof}

If $\frac{n}{c}>8$, we have
$\frac{sn}{bc}\geq\frac{sn}{3(s+1)c}>2$ (for if $s\leq 2, \frac{n}{c}>12$),
so by Lemma~\ref{L4.2} $\ind(S)=1$.

Next we assume that $4<\frac{n}{c}<8$. We divide our proof into two cases.

{\bf Case 1.}\quad $\frac{n}{c}<m<\frac{n}{b}$ for some integer $m \leq 8$.
If $\gcd(m,n)=1$, by Lemma~\ref{L4.1},  $\ind(S)=1$. If $\gcd(m,n)>1$, since $m\leq 8$, we must have $m=p=7$, so $n\geq 7^3=343$. If $\frac{n}{c}<6$, we may take $m=6$ and thus $\ind(S)=1$.  Thus we may assume $\frac{n}{c}<7<\frac{n}{b}$. Since $n<7c \leq 21(s+1)$, we have $s\geq 16$.
Since $\frac{sn}{bc}>\frac{7s}{3(s+1)}\geq \frac{112}{51}>2$,
it follows from Lemma~\ref{L4.2} that $\ind(S)=1$.

{\bf Case 2.}\quad $m<\frac{n}{c}<\frac{n}{b}<m+1$ with $m\in[4,7]$.  If $\frac{sn}{bc}> 2$, By Lemma~\ref{L4.2}  we have $\ind(S)=1$. So we may assume that
$\frac{sn}{bc}\leq 2$. Since $s\geq 12$, we have $\frac{kn}{b}-\frac{kn}{c}=\frac{kn}{bc}>\frac{4k}{3(s+1)}\geq\frac{4(s-1)}{3(s+1)}>1$ for $k=s-1, s$. Thus there exist two integers $M_{s-1} \in [\frac{(s-1)n}{c},\frac{(s-1)n}{b}]$ and $M_s\in [\frac{sn}{c},\frac{sn}{b}]$.  Note that $|M_s-M_{s_1}|\leq2+ \frac{sn}{c}-\frac{(s-1)n}{b}= 2+\frac{(c-s)n}{bc}\leq 2+\frac{(2s+3)n}{bc}\leq 2+ 4+\frac{3n}{bc}<6+\frac{n}{12c}<7$. We conclude that at least one of $M_{s-1}, M_s$ is cop-prime to $n$. Since $M_sa <n$, it follows from Lemma~\ref{L4.1} that $\ind(S)=1$.
\end{proof}

In terms of Lemma~\ref{L4.7} and its proof, we may assume that $m<\frac{n}{c}<\frac{n}{b}<m+1$ with $m\in [2,3]$.

\begin{lemma}\label{L4.8}
If $S$ is a minimal zero-sum sequence such that  $3<\frac{n}{c}<\frac{n}{b}<4$,
then  $\ind(S)=1$.
\end{lemma}

\begin{proof}
Suppose that $n=3c+j$ and $n=3b+j+3$, where $1\leq j\leq b-4$ and $\gcd(j,3)=1$.
If $j\geq 17$, as in  Case 2 of Lemma~\ref{L4.6} we have
\beqs \frac{kn}{b}-\frac{kn}{c}=\frac{kn}{bc}\geq\frac{(b-5)(3b+j+3)}{3b(b+1)}>1,\eeqs
for $k=s-1,s$. By Lemma~\ref{L4.2} we have $\ind(S)=1$.

Next assume that $j\leq 16$. Note that
\beqs S&=&(g)\cdot(g)\cdot(\frac{n-j}{3}g)\cdot(\frac{n+j+3}{3}g)((n-3)g)\\
&\sim&(3g)\cdot(3g)\cdot((n-j)g)\cdot((j+3)g)\cdot((n-9)g)\quad  (\hbox{\rm multiplied by\;}3).\eeqs Let $S'=3S$. Next we show that $\ind(S)=\ind(S')=1$.

{\bf Case 1.} $8\leq j\leq 16$.  Let $n=(j+1)t+r$ with $r\in[0,j]$. Since $\frac{n}{j+1}\geq\frac{289}{17}=17$, we have $t\geq 17\geq j>r$. If $\gcd(n,t)=1$, we have
\beqs S'&\sim&(3tg)\cdot(3tg)\cdot((2t+r)g)\cdot((t-r)g)\cdot(((j-8)t+r)g)\quad  (\hbox{\rm multiplied by\;}t).\eeqs
Since $(3t)+(3t)+(2t+r)+(t-r)+((j-8)t+r)=(j+1)t+r=n$, we have $\ind(S)=1$.
If $\gcd(n,t)>1$, we have
\beqs S'&\sim&(3(t-1)g)\cdot(3(t-1)g)\cdot((t+r+j)g)\cdot((2t-r-j-3)g)\cdot(((j-8)t+r+9)g)\\
&&\quad  \hbox{\rm multiplied by\;}t-1.\eeqs
Note that $t\geq 17 \geq j+1 \geq r+2$. We always have $2t-r-j-3 > 0$ for all $j\in[8,16]$. Again,  we have
$3(t-1)+3(t=1)+(t+r+j)+(2t-r-j-3)+((j-8)t+r+9)=(j+1)t+r= n$, so $\ind(S)=1$.

{\bf Case 2.} $j=7$.  Let $n=9t+r$ with $r\in[1,8]$.  Then $t\geq 32>r$. If $\gcd(n,t)=1$, we have
\beqs S'&\sim&(3tg)\cdot(3tg)\cdot((2t+r)g)\cdot((t-r)g)\cdot(rg)\\
&&\quad  (\hbox{\rm multiplied by\;}t).\eeqs
Since $(3t)+(3t)+(2t+r)+(t-r)+r=9t+r=n$, we have  $\ind(S)=1$.

If $\gcd(n,t)>1$, then  $\gcd(n,t-1)=1$. So we have
\beqs S'&\sim&(3(t-1)g)\cdot(3(t-1)g)\cdot((2t+r+7)g)\cdot((t-r-10)g)\cdot((r+9)g)\\
&&\quad  (\hbox{\rm multiplied by\;}t-1).\eeqs
Since $(3(t-1))+(3(t-1))+(2t+r+7)+(t-r-10)+(r+9)=9t+r=n$, we again  have  $\ind(S)=1$.

{\bf Case 3.} $j=5$.  Let $n=5t+r$ with $r\in[1,4]$.  Then $\gcd(n,t)=1$ and $t\geq 57>3r$. We have
\beqs S'&\sim&((t-r)g)\cdot((t-r)g)\cdot(2rg)\cdot((t-3r)g)\cdot((2t+4r)g)\\
&&\quad  (\hbox{\rm multiplied by\;}2t).\eeqs
Since $(t-r)+(t-r)+(2r)+(t-3r)+(2t+4r)=5t+r=n$, we have $\ind(S)=1$.

{\bf Case 4.} $j=4$.  Let $n=7t+r$ with $r\in[0,6]$.  If $r\not=0$, we have  $\gcd(n,t)=1$
\beqs S'&\sim&(6tg)\cdot(6tg)\cdot((6t+2r)g)\cdot((7t-r)g)\cdot((3t+3r)g)\\
&&\quad  (\hbox{\rm multiplied by\;}2t)\\
&\sim&((t+r)g)\cdot((t+r)g)\cdot((t-r)g)\cdot(2rg)\cdot((4t-2r)g)\\
&&\quad  (\hbox{\rm multiplied by\;}n-1=7t+r-1).\eeqs
Since $(t+r)+(t+r)+(t-r)+(2r)+(4t-2r)=7t+r=n$,  $\ind(S)=1$.

If $r=0$, we have  $p=7$ and $\gcd(n,t-1)=1$
\beqs S'&\sim&(6(t-1)g)\cdot(6(t-1)g)\cdot((6t+2r+8)g)\cdot((7t-r-14)g)\cdot((3t+3r+18)g)\\
&&\quad  (\hbox{\rm multiplied by\;}2(t-1))\\
&\sim&((t+r+6)g)\cdot((t+r+6)g)\cdot((t-r-8)g)\cdot((2r+14)g)\cdot((4t-2r-18)g)\\
&&\quad  (\hbox{\rm multiplied by\;}n-1=7t+r-1).\eeqs
Again, we have $\ind(S)=1$.

{\bf Case 5.} $j=2$.  Let $n=5t+r$ with $r\in[1,4]$.  Then $\gcd(n,t)=1$. We have
\beqs S'&\sim&(tg)\cdot(tg)\cdot((t+r)g)\cdot((5t-r)g)\cdot((2t+2r)g)\\
&&\quad  (\hbox{\rm multiplied by\;}2t)\\
&\sim&(4tg)\cdot(4tg)\cdot((4t+4r)g)\cdot((5t-4r)g)\cdot((3t+7r)g)\\
&&\quad  (\hbox{\rm multiplied by\;}4)\\
&\sim&((t+r)g)\cdot((t+r)g)\cdot((t-3r)g)\cdot((4r)g)\cdot((2t-6r)g)\\
&&\quad  (\hbox{\rm multiplied by\;}n-1=5t+r-1).\eeqs
Since $(t+r)+(t+r)+(t-3r)+(4r)+(2t-6r)=5t+r=n$,  $\ind(S)=1$.

{\bf Case 6.} $j=1$.  Let $n=4t+r$ with $r\in[1,3]$.  Then $\gcd(n,t)=1$. We have
\beqs S'&\sim&(3tg)\cdot(3tg)\cdot((3t+r)g)\cdot((4t)g)\cdot((3t+3r)g)\\
&&\quad  (\hbox{\rm multiplied by\;}t)\\
&\sim&((t+r)g)\cdot((t+r)g)\cdot(tg)\cdot(rg)\cdot((t-2r)g)\\
&&\quad  (\hbox{\rm multiplied by\;}n-1=4t+r-1).\eeqs
Since $(t+r)+(t+r)+t+r+(t-2r)=4t+r=n$, we have $\ind(S)=1$.
\end{proof}

\begin{lemma}\label{L4.9}
 If $S$ is a minimal zero-sum sequence such that $2<\frac{n}{c}<\frac{n}{b}<3$,
then  $\ind(S)=1$.
\end{lemma}

\begin{proof}
We will show that there exist integers $x, y\in [1, \lfloor\frac{b}{3}\rfloor]$ such that
\begin{center}
$\frac{n}{c}  < 2+ \frac{x}{y} < \frac{n}{b}$.
\end{center}
Then $(2y+x)a  < \frac{yn}{b} 3\time 3 \le n$. If $\gcd(2y+x, n)=1$, then by Lemma~\ref{L4.1} $\ind(S)=1$ and we are done.


Suppose $n=2b+b_0$, where $1 \le b_0 \le b-1$.  Since $n$ is prime power and $c = b+1 < \frac{n-1}{2}= b + \frac{b_0-1}{2}$, we infer that
\begin{center}
$b_0 \equiv 1 \pmod 2$ and $b_0 > 3$.
\end{center}
It suffices to show there exist $x, y\in [1, \lfloor\frac{b}{3}\rfloor]$ such that \begin{center}
$\frac{b_0-2}{b+1}  < \frac{x}{y} < \frac{b_0}{b}$ and $\gcd(2y+x, n)=1$.
\end{center}

If it is not easy to determine whether or not  $\gcd(2y+x, n)=1$,  we will show that there exist  two more integers $ z, w \in [1, \lfloor\frac{b}{3}\rfloor]$ such that \begin{center}  $\frac{b_0-2}{b+1}  < \frac{z}{w}  < \frac{b_0}{b}$,
 and $1 \le |(2y+x)- (2w+z)| \le p-1$. \end{center}

 Thus either $\gcd(2y+x, n)=1$ or $\gcd(2w+y, n)=1 $, again we have $\ind(S)=1$. We divide the proof into three cases.


\medskip
\noindent {\bf Case 1.} $b \equiv 0 \pmod 3$. Since $n$ is prime power, we infer that $b_0 \not\equiv 0 \pmod 3$. Suppose $b=3s$.

If  $b_0=3t+1$, then let $x=t$ and $y=s$. We infer that $\frac{3t-1}{3s+1} < \frac{x}{y} < \frac{3t+1}{3s}$. Since $n=2b+b_0=3(2s+t)+1=3(2y+x)+1$, we have $\gcd(2y+x, n)=1$ and we are done.

If  $b_0=3t+2$, let $x=t$ and $y=s$. Then $\frac{3t}{3s+1} < \frac{x}{y} < \frac{3t+2}{3s}$. Since $n=2b+b_0=3(2s+t)+2=3(2y+x)+2$, we have $\gcd(2y+x, n)=1$ and we are done.

\medskip
\noindent {\bf Case 2.} \quad $b \equiv 1 \pmod 3$. Since $n=p^{\mu}$ with a prime $p \ge 7$, we infer that $b_0 \not\equiv 1 \pmod 3$.  Suppose $b=3s+1$. Since $289\le n < 9s+3$,  $s\ge 32$.

\noindent {\bf Subcase 2.1.} \quad $b_0=3t $. Since $b_0 \equiv 1 \pmod 2$ and $ b_0 > 3$,  $t \ge 3$.

If $s < 2t -2$, let  $x=t-1, \, y=s$. Then $\frac{3t-2}{3s+2} < \frac{x}{y}  < \frac{3t}{3s+1}$. Since $n=2b+b_0=3(2s+t)+2=3(2y+x)+5$, we have $\gcd(2y+x, n)=1$ and we are done.

Next assume that $s \ge 2t-2$. Choose $y= s- \lceil\frac{s-2t+3}{3t-2}\rceil$ and $x=t-1$. Then $\frac{3t-2}{3s+2} < \frac{x}{y} < \frac{3t}{3s+1}$. If $s \ge 3t+1$, let $w=y-1, \, z=t-1$, then $\frac{3t-2}{3s+2} < \frac{z}{w} < \frac{3t}{3s+1}$ and  $|(2y+x)- (2w+z)| =2$,  so we are done. Now assume that $$2t-2 \le s \le 3t.$$ Since $32 \le s \le 3t$, $t \ge 11$.
If $ s \le \frac{11(t-1)}{4}$, let  $y= s- \lceil\frac{s-2t+3}{3t-2}\rceil= s-1 $  and $x=t-1$ as before.  Let $w = s-  \lceil\frac{4s-2t+5}{3t-2}\rceil = s-3$
and $z=t-2$. Then $\frac{3t-2}{3s+2} < \frac{x}{y}, \frac{z}{w} < \frac{3t}{3s+1}$ and $|(2y+x)- (2w+z)| =5$,  so we are done. If $  \frac{11(t-1)}{4} < s \le 3t$, let  $y= s- 2 $, $x=t-1$,  $w =  s-4$ and $z=t-2$. Then $\frac{3t-2}{3s+2} < \frac{x}{y}, \frac{z}{w} < \frac{3t}{3s+1}$ and $|(2y+x)- (2w+z)| =5$,  so we are done.

\noindent {\bf Subcase 2.2.}\quad $b_0=3t+2$. Let $x=t$ and $y=s$. We infer that $\frac{3t}{3s+2} < \frac{x}{y} < \frac{3t+2}{3s+1}$. Since $n=2b+b_0=3(2s+t)+4=3(2y+x)+4$, we have $\gcd(2y+x, n)=1$ and we are done.

\medskip
\noindent {\bf Case 3.}\quad $b \equiv 2 \pmod 3$. Since $n=p^{\mu}$ with a prime $p \ge 7$, we infer that $b_0 \not\equiv 2 \pmod 3$.  Suppose $b=3s+2$. Since $289 \le 3b-1=9s+5$, $s\ge 32$.

\medskip
\noindent {\bf Subcase 3.1.} \quad $b_0 \equiv 0 \pmod 3$. Suppose  $b_0=3t$. Recall that $b_0=3t  \equiv 1 \pmod 2$ and $b_0>3$, we infer that  $t \ge 3$ and $ t \equiv 1 \pmod 2$.

If $s < \frac{3t -4}{2} $, then  let $x=t-1, \,y=s; \, z=t-2, \, w=s-1$. We infer that $\frac{3t-2}{3s+3} <\frac{z}{w} <  \frac{x}{y} < \frac{3t}{3s+2}$ and $|(2y+x)- (2w+z)| =3 $, so we are done. Next assume that $s \ge \frac{3t -4}{2}$. If $s < 3t-3$, then $t\ge 13$. Let $x=t-1, \,y=s; \, z=t-2, \, w=s-2$. Then $\frac{3t-2}{3s+3}<  \frac{x}{y} , \frac{z}{w} < \frac{3t}{3s+2}$ and $|(2y+x)- (2w+z)| =5 $, so we are done.

If $s \ge 3t-3$, choose $y= s- \lceil\frac{s-3t+4}{3t-2}\rceil$ and $x=t-1$. Then $\frac{3t-2}{3s+3} < \frac{x}{y} < \frac{3t}{3s+2}$.   Let $w=y-1, z=t-1$.  Note that if $s< 3t$ then $t> 11$. Thus we have  either $s \ge 3t$ or $t \ge 9$. Therefore, $\frac{3t-2}{3s+3} < \frac{z}{w} < \frac{3t}{3s+2}$. 
Since $|(2y+x)- (2w+z)| =2 $, we are done. 

\medskip
\noindent {\bf Subcase 3.2.} \quad $b_0 \equiv 1 \pmod 3$. Suppose  $b_0=3t+1$. Recall that $b_0=3t +1  \equiv 1 \pmod 2$ and $b_0 >3$. Hence $t \equiv 0 \pmod 2$ and thus $t \ge 2$.

If $s > 5t+1$, let $x=t, \, y=s; \, z=t, \, w= s-1$. Then $\frac{3t-1}{3s+3}<  \frac{x}{y} <\frac{z}{w} < \frac{3t+1}{3s+2}$ and $|(2y+x)- (2w+z)| =2 $,  so we are done. Hence we may assume that $s \le 5t+1$.

If $s \ge 3t- 2$, let  $x= z = t-1, y = s-\lceil\frac{2s-3t+4}{3t-1}\rceil, w=y-1$. If $t \ge 10$, then $\frac{3t-1}{3s+3}<  \frac{z}{w} , \frac{x}{y}  < \frac{3t+1}{3s+2}$ and $ |(2y+x)- (2w+z)| = 2$, so we are done. Hence we may assume that $t \le 9$ and $3t-2 \le s \le 5t+1$. Note that $t \equiv 0 \pmod 2$ and $s \ge 33$. Hence $t=8$ and then $33 \le s \le 41$.  Let  $x= z = t-1, y = s-3, w=s-4$. Then $\frac{3t-1}{3s+3}<  \frac{z}{w} , \frac{x}{y}  < \frac{3t+1}{3s+2}$ and $ |(2y+x)- (2w+z)| = 2$, so we are done.

If $2t < s < 3t-2$, let $x=t, \, y=s; \, z=t-1, \, w= s-1$. Then $\frac{3t-1}{3s+3}<  \frac{x}{y} , \frac{z}{w} < \frac{3t+1}{3s+2}$ and $|(2y+x)- (2w+z)| = 3$, so we are done.

Hence we may assume that $s \le 2t$. Note that $t \ge \frac{s}{2}\ge \frac{33}{2}$, so $t\ge 17.$


If $s < \frac{6t-7}{5}$, let $x=t-1, \, y=s; \, z=t-2, \, w= s-1$. Then $\frac{3t-1}{3s+3}<  \frac{z}{w} <\frac{x}{y} < \frac{3t+1}{3s+2}$ and $|(2y+x)- (2w+z)| =3 $, so we are done. Hence we may assume that $s \ge \frac{6t-7}{5}$.

If $s < \frac{3t-3}{2}$, let $x=t-1, \, y=s; \, z=t-2, \, w= s-\lceil\frac{5s-3t+3}{3t-1}\rceil$. Then $\frac{3t-1}{3s+3}<  \frac{z}{w} , \frac{x}{y} < \frac{3t+1}{3s+2}$ and $1 \le |(2y+x)- (2w+z)| \le 5 $, so we are done. Hence we may assume that $s \ge \frac{3t-3}{2}$.

Next assume that $\frac{3t-3}{2} \le s \le 2t$.

Let $x=t-1$ and $y=s-1$. We infer that $\frac{3t-1}{3s+3} < \frac{t-1}{s-1} < \frac{3t+1}{3s+2}$. Let $z=t-2, \, w= s-\lceil\frac{5s-3t+3}{3t-1}\rceil$. Then $\frac{3t-1}{3s+3}<  \frac{z}{w} , \frac{x}{y} < \frac{3t+1}{3s+2}$ and $1 \le |(2y+x)- (2w+z)| \le 5 $, so we are done.
\end{proof}

Now Proposition~\ref{P3.3} follows immediately from Lemmas~\ref{L4.7}, \ref{L4.8}, \ref{L4.9} and Proposition~\ref{P2.2}.

We remark that the case when $\mathsf h(S)=1$ is much more complicated and  $\ind(S)$ is not yet determined (even when $n=p$ is a prime number). We conclude the paper by listing this case  as an open problem.

\noindent{\bf Open Problem}. Let $S$ be a minimal zero-sum sequence of length five over a cyclic group of order $n$. Determine $\ind(S)$ when the height $\mathsf h(S)=1$.

\bigskip

\begin{center}
{ACKNOWLEDGEMENTS}
\end{center}

\vskip .6 cm
The first author is  supported  by the  NNSF of China (Grant No. 11001110 and 11271131) and {\it Jiangsu Government Scholarship for Overseas Studies}. The second author is supported by a Discovery Grant from the Natural Science and Engineering Research Council of Canada. The third  author is  supported  by the  NSFC (Grant No. 11271207 and 11301531)

\bigskip

\vskip30pt
\def\refname{

\end{document}